\documentclass[10pt,a4paper]{amsart}
\usepackage[utf8]{inputenc}
\usepackage[T1]{fontenc}
\usepackage{amsmath}
\usepackage{upquote}
\usepackage{bm}
\usepackage{amsfonts}
\usepackage{amssymb}
\usepackage[english]{babel}
\usepackage{tabularx}
\usepackage{mathtools}
\usepackage{adjustbox}
\usepackage{graphics}
\usepackage{hyperref}
\usepackage{xcolor}
\usepackage{diagbox}
\usepackage{color}
\usepackage{tikz}
\usepackage{chngcntr}
\usepackage{array}
\newcolumntype{H}{>{\setbox0=\hbox\bgroup}c<{\egroup}@{}}
\usetikzlibrary{calc}
\numberwithin{equation}{section}
\newtheorem{thm}{Theorem}[section]
\newtheorem{pr}[thm]{Proposition}
\newtheorem{lm}[thm]{Lemma}
\newtheorem{cor}[thm]{Corollary}
\theoremstyle{remark}
\newtheorem{re}[thm]{Remark}

\newtheorem{ex}[thm]{Example}

\newtheorem{pro}[thm]{Problem}

\newtheoremstyle{case}{}{}{}{}{}{:}{ }{}
\theoremstyle{case}

\counterwithin*{case}{thm}
\newtheoremstyle{caso}{}{}{}{}{}{:}{ }{}
\theoremstyle{caso}

\def\op{\operatorname}
\DeclareRobustCommand{\svdots}{
  \vbox{%
    \baselineskip=0.33333\normalbaselineskip
    \lineskiplimit=0pt
    \hbox{.}\hbox{.}\hbox{.}%
    \kern-0.2\baselineskip
  }%
}

\newcommand{\lcm}{\operatorname{lcm}}

\newcommand{\stirlingone}[2]{\genfrac{[}{]}{0pt}{}{#1}{#2}}

\newenvironment{proof*}[1]
  {\renewcommand{\proofname}{#1}%
   \begin{proof}}
  {\end{proof}}

\makeatletter
\let\@@pmod\pmod
\DeclareRobustCommand{\pmod}{\@ifstar\@pmods\@@pmod}
\def\@pmods#1{\mkern4mu({\operator@font mod}\mkern 6mu#1)}
\makeatother

\title[Polynomization of the Bessenrodt--Ono type inequalities]{Polynomization of the Bessenrodt--Ono type inequalities for 
$A$-partition functions}
\author{Krystian Gajdzica}
\address{Institute of Mathematics \\
	Faculty of Mathematics and Computer Science \\
	Jagiellonian University in Cracow
}
\email{krystian.gajdzica@doctoral.uj.edu.pl}

\author{Bernhard Heim}
\address{Faculty of Mathematical and Natural Sciences\\
Mathematical Institute\\
University of Cologne\\
Weyertal 86–90\\
50931 Cologne\\
Germany
}
\email{bheim@uni-koeln.de}

\address{Lehrstuhl A f\"ur Mathematik\\
RWTH Aachen University\\
52056 Aachen\\
Germany
}
\email{bernhard.heim@rwth-aachen.de}

\author{Markus Neuhauser}
\address{Kutaisi International University\\
5/7\\
Youth Avenue\\
Kutaisi\\
4600 Georgia
}
\email{}
\address{Lehrstuhl A f\"ur Mathematik\\
RWTH Aachen University\\
52056 Aachen\\
Germany
}
\email{markus.neuhauser@kiu.edu.ge}



\begin{document}

\setlength{\parindent}{10mm}
\maketitle

\begin{abstract}
For an arbitrary set or multiset $A$ of positive integers, we associate the $A$-partition function $p_A(n)$ (that is the number of partitions of $n$ whose parts belong to $A$). We also consider the analogue of the $k$-colored partition function, namely, $p_{A,-k}(n)$. Further, we define a family of polynomials $f_{A,n}(x)$ which satisfy the equality $f_{A,n}(k)=p_{A,-k}(n)$ for all $n\in\mathbb{Z}_{\geq0}$ and $k\in\mathbb{N}$. This paper concerns the polynomization of the
Bessenrodt--Ono type inequality for  $f_{A,n}(x)$:
\begin{align*}
    f_{A,a}(x)f_{A,b}(x)>f_{A,a+b}(x),
\end{align*}
where $a$ and $b$ are arbitrary positive integers; and delivers some efficient criteria for its solutions. Moreover, we also investigate a few basic properties related to both functions $f_{A,n}(x)$ and $f_{A,n}'(x)$.
\end{abstract}

\section{Introduction}
Let $A$ be a multiset of positive integers and finite mutiplicities $\mu(a)$.
We consider a sequence of polynomials $\{f_{A,n}(x)\}_n$:
\begin{equation} \label{def:polynomization}
\sum_{n=0}^{\infty} f_{A,n}(x) \, q^n := 
\prod_{a \in A} \left( \frac{1}{1-q^a}\right)^x.
\end{equation}
The $n$th polynomial evaluated at positive integers $k$
provide the $k$-colored $A$-partition function $p_{A,-k}(n) = f_{A,n}(k)$. 
Interesting multisets are given by $A_I:=\{1\}$, 
$A_{II}:= \{1, \ldots,1\}$ with $\mu(1)=m$, 
$A_{III}:= \mathbb{N}$, and 
$A_{IV}:=\{1,2,2,3,3,3, 4 \ldots \}$ with $\mu(a)=a$.
Therefore, $p_{A_{III}, -1}= f_{A_{III},n}(1) = \mathrm{p}\left( n\right) $
the usual partition
function, and more general $p_{A_{III}, -k}= f_{A_{III},n}(k)$ the $k$-colored partition function. Note that 
$$f_{A_{I},n}(x) = \frac{1}{n!}
\sum_{k=0}^{n} \stirlingone{n}{k}x^k,$$ where the coefficients are the unsigned
Stirling numbers of the first kind. Further, $f_{A_{II},n}(x) = f_{A_I,n}( m \, x)$
and $f_{A_{IV},n}(1)=\op{pp}\left( n\right) $, where $\op{pp}\left( n\right) $
denotes the plane partition function.

In this paper we prove the following Bessenrodt--Ono
type \cite{BO} of inequalities. Let $a,b \in \mathbb{N}$ with $a+b >2$. Let $A$ be a multiset with $1 \in A$. We have
\begin{equation}
f_{A,a}(x) \, f_{A,b}(x) > f_{A,a+b}(x)  \label{BO}
\end{equation}
\begin{itemize}
\item[(i)] for $ x \geq 3$ and $A$ an ordinary set. Note that $f_{A,1}(3)^2 - f_{A,2}(3)=0$. We refer to Theorem \ref{th:set}.
\item[(ii)] for $ x \geq 5$ and $A$ a multiset with $\mu(a) \leq a$ for all all $a \in A$. Note that 
$f_{A,1}(5)^2 - f_{A,2}(5)=0$. We refer to Theorem \ref{th:multiset}.
\end{itemize}
Note that (ii) implies (i) for $x \geq 5$. Moreover, let $A$ be a set, $B$ the multiset derived from $A$, where the elements of $B$ are the elements of $A$ with multiplicity $m$. Then
\begin{equation}
f_{B,n}(x) = f_{A,n}(m \, x).
\end{equation}

Let us recall some basic definitions and known results.
For a non-negative integer $n$, the partition function
$\mathrm{p}\left( n\right) $ enumerates all possible sequences $$\lambda=(\lambda_1,\lambda_2,\ldots,\lambda_j)$$ of positive integers such that $\lambda_1\geq\lambda_2\geq\cdots\geq\lambda_j\geq1$ and 
\begin{equation*}
    n=\lambda_1+\lambda_2+\cdots+\lambda_j.
\end{equation*}
Elements $\lambda_i$ are called the parts of the partition $\lambda$. For example, there are $5$ partitions of $4$: $(4)$, $(3,1)$, $(2,2)$, $(2,1,1)$ and $(1,1,1,1)$. Thus, $\mathrm{p}\left( 4\right) =5$.
It should be pointed out that
$\mathrm{p}\left( n\right) =0$ if $n$ is negative, and
$\mathrm{p}\left( 0\right) =1$---as
there is only the empty partition in this case. In the 1740's Euler discovered the generating function for $\mathrm{p}\left( n\right) $
\begin{equation*}
    \sum_{n=0}^\infty \mathrm{p}\left( n\right) q^n=\prod_{i=1}^\infty\frac{1}{1-q^i}.
\end{equation*}
There is a wealth of literature devoted to the partition theory. For a general introduction to the subject, we encourage the reader to see Andrews' books \cite{GA2, GA1} as well as \cite{AK,H,Ono1,Sills}.

However, it also motivated 
to investigate other similar problems for combinatorial sequences. For instance, in 2016 Bessenrodt and Ono \cite{BO} proved the following result.

\begin{thm}[Bessenrodt, Ono]\ \\
    \label{thm0.1}Let $a$ and $b$ be natural numbers such that $a+b>9$. Then
    \begin{equation*}
        \mathrm{p}\left( a\right) \mathrm{p}\left( b\right) >\mathrm{p}\left( a+b\right) .
    \end{equation*}
\end{thm}
Their proof is based on asymptotic estimates due to Rademacher
\cite{R} and Lehmer \cite{L}. It is worth noting that there are also other proofs. Alanazi, Gagola,
and Munagi \cite{AGM} show the above theorem in a combinatorial manner. On the other hand, Heim and Neuhauser \cite{HN3} present the reasoning via induction on $a+b$ to obtain the aforementioned property.

It should be pointed out that Theorem \ref{thm0.1} is not only art of art's sake, but it allows us to determine the value of

\begin{equation*}
    \max{\mathrm{p}\left( n\right) }=\max{\left( \mathrm{p}\left( \bm{\lambda}\right) :\bm{\lambda} \text{ is a partition of } n\right)},
\end{equation*}
where $\mathrm{p}\left( \bm{\lambda}\right) =\prod_{i=1}^{j}\mathrm{p}\left( \lambda _{i}\right) $ for 
$\bm{\lambda}=(\lambda_1,\lambda_2,\ldots,\lambda_j)$. For more details, we encourage to see Bessenrodt and Ono's article \cite{BO}.

Since 2016, 
a lot of people have been examining analogous inequalities for various types of partition functions, among others, Chern, Fu,
and Tang \cite{Chern}; Beckwith and Bessenrodt \cite{BB}; Hou
and Jagadeesan \cite{Hou}; Males \cite{Males}; Dawsey and Masri \cite{Dawsey}; Heim, Neuhauser,
and Tr\"{o}ger \cite{HNT2}; Gajdzica \cite{KG2}. 

However, it turns out that we can also pass from the discrete task to the
continuous one by considering the appropriate family of polynomials---that
is \linebreak Gian–Carlo Rota’s advice to study problems in combinatorics and number theory in terms of roots of polynomials. In 1955, Newman \cite{Newman} began studying a family of polynomials $P_n(x)$ with exceptional properties. These polynomials are defined via the recursive formula $P_0(x) := 1$ and
\begin{equation*}
    P_n(x):=\frac{x}{n}\sum_{j=1}^n\sigma(j)P_{n-j}(x)
\end{equation*}
for $n\geq1$, where $\sigma(m):=\sum _{\ell \mid m}\ell $
is the sum of divisors of $m\in\mathbb{N}$. Each polynomial $P_n(x)$ is of degree $n$, and is integer-valued. Further, these functions were investigated by Heim et al. \cite{HN4, HN5, HN6}. In particular, they exhibited a practical formula for the derivatives $P_n'(x)$ for all $n\in\mathbb{N}$ \cite{HN4}, namely,
\begin{equation*}
    P_n'(x)=\sum_{j=1}^n\frac{\sigma(j)}{j}P_{n-j}(x).
\end{equation*}
Moreover, Heim and Neuhauser \cite{HN5} showed that the equality
\begin{equation*}
    P_n(k)=p_{-k}(n)
\end{equation*}
is true for all positive integers $k$ and $n$, where
$p_{-k
}\left( n
\right) $ denotes the number of $k$-colored partitions of $n$---in
other words the generating function corresponding to $p_{-k}(n)$ is given by
\begin{equation*}
    \sum_{n=0}^\infty p_{-k}(n)q^n=\prod_{i=1}^\infty\frac{1}{(1-q^i)^k}.
\end{equation*}
Furthermore, Heim, Neuhauser, and Tr\"{o}ger \cite{HN6} proved the
Bessenrodt--Ono type inequality for polynomials $P_n(x)$. More specifically, they obtained the following result.
\begin{thm}[Heim, Neuhauser, Tr\"{o}ger] \ \\
    \label{thm0.2}Let $a$ and $b$ be positive integers such that $a+b>2$, and let $x>2$. Then, we have that 
    \begin{equation*}
        P_a(x)P_b(x)>P_{a+b}(x).
    \end{equation*}
    If $x=2$, then the above inequality is valid for all $a+b>4$.
\end{thm}
It should be pointed out that there are more results of that kind for other families of polynomials. For example, Heim, Neuhauser,
and Tr\"{o}ger \cite{HNT2} considered a family polynomials associated to the so-called plane partition function
$\op{pp}\left( n\right) $. That is the number of partitions of $n$ such that
each part $j$ might appear in $j$ distinct colors. The generating function for
$\op{pp}\left( n\right) $ takes the form
\begin{equation*}
    \sum_{n=0}^\infty \op{pp}\left( n\right) q^n=\prod_{i=1}^\infty\frac{1}{(1-q^i)^i}.
\end{equation*}
For more information about plane partitions and the plane partition function, we refer the reader to \cite[Section~12]{GA2}, \cite[Section 7.20]{Stanley},
and \cite{K}. Further, Li \cite{Li} obtained the
Bessenrodt--Ono type inequalities for the overpartition function $\overline{p}(n)$ and its polynomization. Let us recall that the overpartition function $\overline{p}(n)$ counts the number of partitions of $n$ such that the last occurrence of each distinct part may be overlined. In that case the corresponding generating function is given by
\begin{equation*}
    \sum_{n=0}^\infty\overline{p}(n)q^n=\prod_{i=1}^\infty\frac{1+q^i}{1-q^i}.
\end{equation*}
For more details related to overpartitions and the overpartition function, see \cite{Corteel}.

In this paper we deal with the Bessenrodt--Ono
type inequalities for the so-called $A$-partition function and its polynomization. Let $A$ be an arbitrary (finite or infinite) set or multiset of positive integers. By an $A$-partition of a non-negative integer $n$, we mean any partition $\lambda=(\lambda_1,\lambda_2,\ldots,\lambda_j)$ of $n$ such that $\lambda_i\in A$ for each $1\leq i\leq j$. For the sake of clarity, we also assume that two $A$-partitions of $n$ are considered the same if there is only a difference in the order of their parts. The $A$-partition function $p_A(n)$ enumerates all possible $A$-partitions of $n$. Additionally, we set $p_A(n)=0$ if $n$ is negative, and $p_A(0)=1$. For example, if $A=\{1,2,\textcolor{blue}{2},3,\textcolor{blue}{3}\}$, then $p_A(4)=8$ with: $(\textcolor{blue}{3},1)$, $(3,1)$, $(\textcolor{blue}{2},\textcolor{blue}{2})$, $(\textcolor{blue}{2},2)$, $(2,2)$, $(\textcolor{blue}{2},1,1)$, $(2,1,1)$,
and $(1,1,1,1)$. The generating function for $p_A(n)$ satisfies the equality
\begin{equation}
    \sum_{n=0}^\infty p_A(n)q^n=\prod_{a\in A}\frac{1}{1-q^a}.
\end{equation}
There is an abundance of literature devoted to $A$-partitions, we refer the reader to, for instance, \cite{GA, Beck, Bell, CN, DV, KG1, Gawron, RS, MU2, MUŻ}.

It is proven in the next section that one can assign unique polynomials $f_{A,n}(x)$ to the given $A$-partition function in such a way that $f_{A,n}(k)=p_{A,-k}(n)$ holds for every natural number $k$ and non-negative integer $n$, where $p_{A,-k}(n)$ is, similarly to the case of the classical partition function, the $k$-colored $A$-partition function.

Further, we proceed to the main part of the paper. Using the law of induction and the analogous methods to those presented in Heim, Neuhauser,
and Tr\"{o}ger's paper \cite{HN6}, we prove the following result.   

\begin{thm} \label{th:set}
    For any set $A\subseteq\mathbb{N}$ with $1\in A$, the inequality
    \begin{equation*}
        f_{A,a}(x)f_{A,b}(x)>f_{A,a+b}(x)
    \end{equation*}
    holds for all $x\geq3$ and $a,b\in\mathbb{N}$ except $a=b=1$. For $a=b=1$, we have that 
    $$f_{A,1}(3)f_{A,1}(3)\geq f_{A,2}(3)$$
    with
equality whenever $2\in A$ and 
    $$f_{A,1}(x)f_{A,1}(x)>f_{A,2}(x)$$
    for all $x>3$.
\end{thm}

It is worth pointing out that the above property significantly extends Theorem \ref{thm0.2}---in
a sense that instead of taking $\mathbb{N}$, we can consider any finite or infinite set. For example, one may apply it to obtain the polynomization of the
Bessenrodt--Ono type inequality for the $k$-regular partition function for every $k\geq2$. That is an extension of Beckwith and Bessenrodt's theorem, see \cite[Theorem 2.1]{BB}.

Moreover, we also show a similar result for some family of multisets $A$. More precisely, we figure out the following.

\begin{thm}\label{th:multiset}
     If $A$ is such a multiset of natural numbers that $1$ occurs exactly once
in $A$, and each number $j$ appears at most $j$ times in $A$, then the inequality
    \begin{equation*}
        f_{A,a}(x)f_{A,b}(x)>f_{A,a+b}(x)
    \end{equation*}
    holds for all $x\geq5$ and $a,b\in\mathbb{N}$ except $a=b=1$. For $a=b=1$, we have that 
    $$f_{A,1}(5)f_{A,1}(5)\geq f_{A,2}(5)$$
    with
equality whenever $2$ appears exactly two times in $A$ and 
    $$f_{A,1}(x)f_{A,1}(x)>f_{A,2}(x)$$
    for every $x>5$.
\end{thm}

On the other hand, the above theorem generalizes Heim, Neuhauser,
and Tr\"{o}ger's result devoted to the plane partition function \cite[Theorem 1.3]{HNT2}.

This manuscript is organized as follows. In Section 2, we introduce necessary notations, define the appropriate family of polynomials $f_{A,n}(x)$ associated to the fixed $A$-partition function and show some of their basic properties. Section 3 deals with the
Bessenrodt--Ono inequalities for both $k$-colored $A$-partition functions $p_{A,-k}(n)$ and the polynomials $f_{A,n}(x)$ for such sets $A$ that $1\in A$ ($\mu(a)=1$ for each $a\in A$). Finally, Section 4 examines the analogues problems to those from Section 3 for such multisets $A$ that $1\in A$ and $\mu(j)\leq j$ for every $j\in\mathbb{N}.$

\section{Preliminaries}

At first, let us fix some notations and conventions. By $\mathbb{R}$,
$\mathbb{N}$, and $\mathbb{Z}_{\geq0}$ we denote
the set of real numbers, the set of positive integers,
and the set of non-negative integers, respectively. For an arbitrary  multiset of positive integers $A$, $\mu(a)$ denotes the number of appearances of the
number $a$ in $A$. Further, for a fixed finite or infinite unbounded set or multiset of positive integers $A$, the $A$-partition function $p_A(n)$ is the number of $A$-partitions of $n$ (actually, if $A$ is infinite, we have to require that $\mu(a)<\infty$ for every $a\in A$, otherwise there will be infinitely many numbers $m$ such that $p_A(m)=\infty$). In such a setting, we also introduce a family of polynomials associated to the $A$-partition function via the power series expansion (\ref{def:polynomization}).

It is straightforward to see that $f_{A,0}(x)=1$ and $f_{A,n}(1)=p_A(n)$. Furthermore, if we consider the $k$-colored $A$-partition function $p_{A,-k}(n)$, then we have that $p_{A,-k}(n)=f_{A,n}(k)$ for all $k\in\mathbb{N}$ and $n\in\mathbb{Z}_{\geq0}$. Moreover, there are recurrence formulae for the polynomials $f_{A,n}(x)$ and $f_{A,n}'(x)$, as well.

\begin{pr}
    For a fixed set (or multiset) $A$ of positive integers, the following identities 
    \begin{equation}\label{1}
    f_{A,n}(x)=\frac{x}{n}\sum_{j=1}^n\sigma_A(j)f_{A,n-j}(x)
\end{equation}
and
\begin{equation}\label{2}
    f_{A,n}'(x)=\sum_{j=1}^n\frac{\sigma_A(j)}{j}f_{A,n-j}(x),
\end{equation}
are true for all $n\in\mathbb{N}$, where $\sigma_A(i)$ is a sum of those divisors of $i$ that belong to $A$.
\end{pr}
\begin{proof}
    If we denote the left hand side of the equality $(\ref{def:polynomization})$ by $F(q,x)$ and compare the appropriate coefficients standing next to $q^n$ of both 
    \begin{align*}
    \frac{\partial F(q,x)}{\partial q}=
    \frac{\partial \log{(F(q,x))}}{\partial q}F(q,x)
\end{align*}
and
\begin{align*}
    \frac{\partial F(q,x)}{\partial x}=
    \frac{\partial \log{(F(q,x))}}{\partial x}F(q,x),
\end{align*}
we obtain the required property.
\end{proof}

Now, it is not difficult to notice that if $1\in A$, then $f_{A,n}(x)$ is a polynomial of degree $n$. Otherwise, we have that $f_{A,1}(x)=0$. In order to omit such a confusion, we assume that henceforth $1\in A$ and $1$ occurs exactly one time in $A$. In addition, that assumption asserts that $\gcd A=1$ which is a necessary and sufficient condition to have the
Bessenrodt--Ono type inequality for $p_A(n)$ when $A$ is a finite set or multiset, see \cite{KG2}. Next, we introduce an additional notation, namely,
\begin{align}
    \Delta_{A,n}(x):=f_{A,n+1}(x)-f_{A,n}(x).
\end{align}
The following characterization follows from the above discussion.

\begin{thm}\label{thm1.1}
    Let $n\in\mathbb{N}$ and $x\in\mathbb{R}$ with $x\geq1$. Then 
    \begin{equation*}
        f_{A,n}(x)\leq f_{A,n+1}(x) \hspace{0.5cm} \text{and} \hspace{0.5cm} 1\leq f_{A,n}'(x)<f_{A,n+1}'(x).
    \end{equation*}
    For $x>1$, we additionally have that   
    \begin{equation*}
        f_{A,n}(x)<f_{A,n+1}(x).
    \end{equation*}
\end{thm}
\begin{proof}
    At first, we show that $f_{A,n}(x)\leq f_{A,n+1}(x)$. The proof is by induction on $n$ under the assumption that $x\geq1$. Let us notice that $f_{A,1}(x)=x$ and $f_{A,2}(x)=x(x+(1+2s_2))/2$, where $s_2$ is the number of $2$s
appearing in the multiset $A$. Thus, we get that 
    \begin{align*}
        \Delta_{A,1}(x)\geq\frac{x}{2}(x-1)\geq0
    \end{align*}
    for every $x\geq1$. Now, we assume that the statement is valid for all $x\geq1$ and $1\leq m\leq n-1$, and check its correctness for $m=n$. The identity (\ref{2}) asserts that
    \begin{equation*}
        f_{A,n+1}'(x)>\sum_{j=1}^n\frac{\sigma_A(j)}{j}f_{A,n+1-j}(x)\geq\sum_{j=1}^n\frac{\sigma_A(j)}{j}f_{A,n-j}(x)>f_{A,n}'(x).
    \end{equation*}
    Thus, we have that the inequality $f_{A,n+1}'(x)>f_{A,n}'(x)$ is satisfied for every $x\geq1$. One can also notice that 
    \begin{equation*}
        f_{A,n+1}(1)=p_A(n+1)=p_A(n)+p_{A\setminus\{1\}}(n+1)\geq f_{A,n}(1).
    \end{equation*}
    Hence, the function $\Delta_{A,n}'(x)$ is positive for every $x\geq1$, and the inequality \linebreak $f_{A,n+1}(x)\geq f_{A,n}(x)$ holds for all such $x$, as required. The second part of the statement follows directly from the above discussion.
\end{proof}

Now, we are ready to investigate the Bessenrodt--Ono
inequality for the polynomials $f_{A,n}(x)$. We divide our consideration into a few section depending on the structure of $A$.

\section{The Bessenrodt--Ono inequality for an arbitrary set $A$ with $1\in A$}

In this section, we assume that $A\subseteq\mathbb{N}$ and $A$ is a
set---meaning that each number $j$ might occur at most once in $A$. Moreover, we also require that $1\in A$. At the beginning, let us consider a very useful example when $A$ is a singleton.

\begin{lm}\label{Lemma0}
    For the singleton $A_I=\{1\}$, the equality
    \begin{equation*}
    f_{A_I,n}(x)=\frac{1}{n!}\prod_{i=0}^{n-1}(x+i)=\binom{x+n-1}{n}
\end{equation*}
holds for every $n\in\mathbb{Z}_{\geq0}$.
\end{lm}
\begin{proof}
    To prove the required property, it is enough to systematically utilize the formula $(\ref{1})$.
\end{proof}
It is worth pointing out that the above polynomials are naturally connected with a well-known combinatorial sequence. The unsigned Stirling number of the first kind $\stirlingone{k}{l}$ is the number of permutations of $k$ elements with exactly $l$ cycles. Equivalently, it may be defined as the appropriate coefficient of the rising factorial (Pochhammer function), namely
\begin{align*}
    x^{\overline{n}}:=x(x+1)\cdots(x+n-1)=\sum_{i=0}^{n}\stirlingone{n}{i}x^i.
\end{align*}
For more information about Stirling numbers, we refer the reader to \cite[Chapter~6]{GKP}. Therefore, it is clear that the equality
\begin{align*}
    f_{A_I,n}(x)=\frac{1}{n!}\sum_{i=0}^{n}\stirlingone{n}{i}x^i
\end{align*}
is valid for all $n\in\mathbb{Z}_{\geq0}$.

Furthermore, let us notice that for any other set (or multiset) $B\subseteq\mathbb{N}$ with $1\in B$, we have that
\begin{equation*}
    f_{A_I,n}(x)\leq f_{B,n}(x)
\end{equation*}
for all $n\geq0$ and $x\geq0$---that
is a direct consequence of the equality (\ref{1}). Moreover, we may easily extend Lemma \ref{Lemma0} to any finite multiset of ones.
\begin{cor}
    For an arbitrary positive integer $m$ and the multiset $A_{II}=\{1,\ldots,1\}$ with $\mu(1)=m$, the equality
    \begin{equation*}
    f_{A_{II},n}(x)=f_{A_I,n}(mx)
\end{equation*}
is satisfied for every $n\in\mathbb{Z}_{\geq0}$.
\end{cor}

Before we go to the main part of this section, let us take a look on a concrete family of polynomials $f_{A,n}(x)$. For instance, if we set $A=\{1,2,3,4,5\}$, then one can easily calculate the corresponding polynomials $f_{A,n}(x)$ for small values of $n$. We present all the cases for $1\leq n\leq7$ in Table 1.
\begin{center}
\begin{table}
\begin{tabular}{ |c|l| } 
\hline
$n$ & $f_{A,n}(x)$\\
 \hline\hline 
 $1$ & $x$ \\ 
 $2$ & $\frac{1}{2}x(x+3)$ \\ 
 $3$ & $\frac{1}{6}x(x+1)(x+8)$  \\ 
 $4$ & $\frac{1}{24} x (x+1) (x+3) (x+14)$  \\ 
 $5$ & $\frac{1}{120} x (x+3) (x+6) \left(x^2+21 x+8\right)$ \\ 
 $6$ & $\frac{1}{720} x (x+2) (x+5) \left(x^3+38 x^2+289 x+72\right)$  \\ 
 $7$ & $\frac{1}{5040}x \left(x^6+63 x^5+1225 x^4+9345 x^3+28294 x^2+25872 x+720\right)$ \\
 \hline
\end{tabular}
\caption{The polynomials $f_{A,n}(x)$ for $A=\{1,2,3,4,5\}$ and $1\leq n \leq7$.}
\end{table}
\end{center}

In order to gain some intuition related to the behaviour of the polynomials $f_{A,a}(x)f_{A,b}(x)-f_{A,a+b}(x)$, one can determine their roots. Figure 1 shows all their complex roots of $1\leq a,b\leq 10$.
\begin{center}
\begin{figure}[!htb]
   \begin{minipage}{1\textwidth}
     \centering
     \includegraphics[width=.9\linewidth]{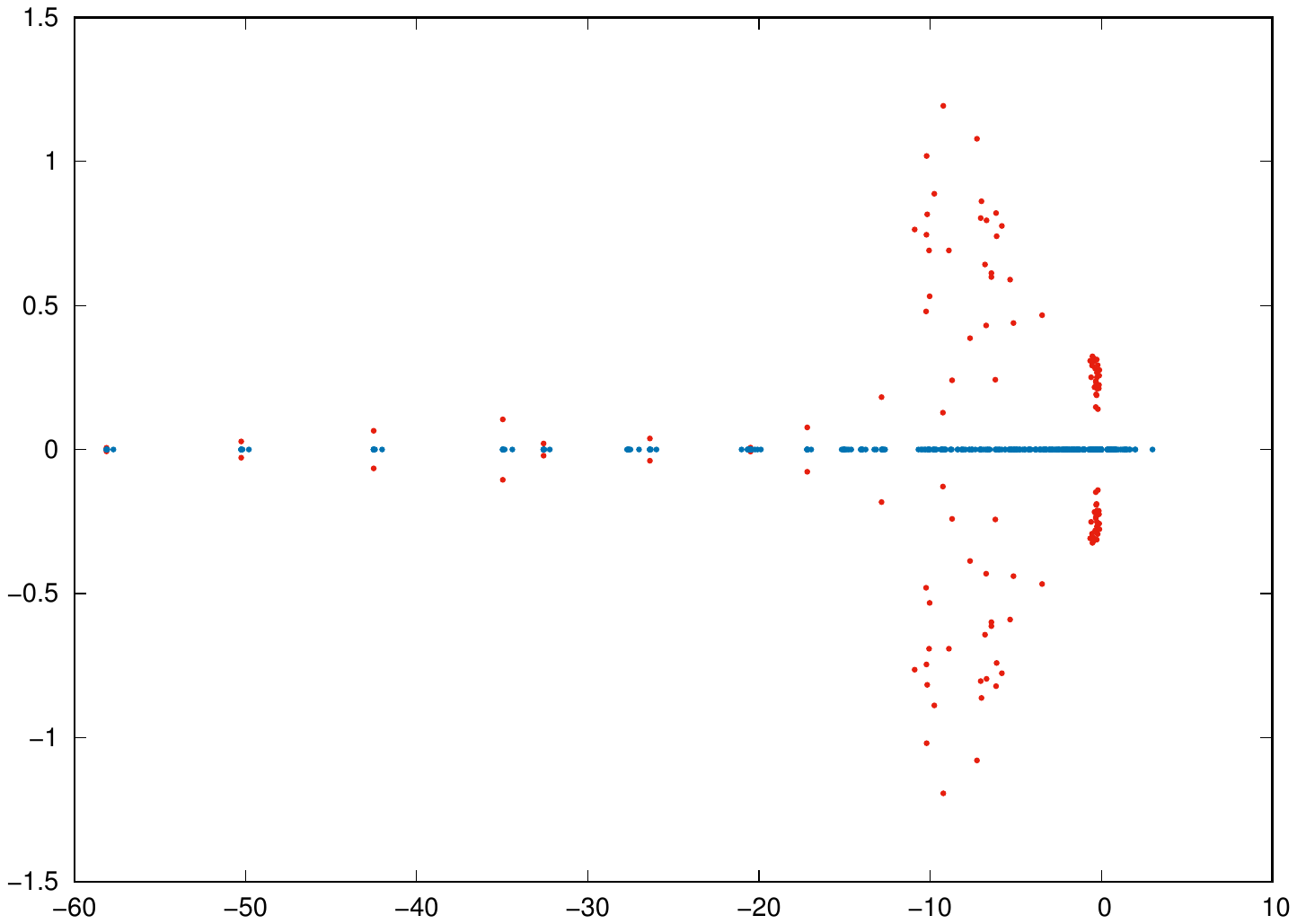}
     \caption{Complex roots of the polynomials $f_{A,a}(x)f_{A,b}(x)-f_{A,a+b}(x)$ for $A=\{1,2,3,4,5\}$ and $1\leq a,b \leq10$.}
   \end{minipage}\hfill
\end{figure}
\end{center}

Now, our aim is to find a Bessenrodt--Ono
type inequality for the $3$-colored $A$-partition function---that
is $p_{A,-3}(n)=f_{A,n}(3)$. That result will play a central role in the proof of the main theorem of this section. To deal with the task, we show the reasoning by induction on $a+b$. Actually, we combine methods recently utilized by Heim,
Neuhauser, and Tr\"oger in \cite{HN3, HN6}.

\begin{thm}\label{thm2.1}
    For every set $A\subseteq\mathbb{N}$ with $1\in A$, the inequality
    \begin{equation*}
        f_{A,a}(3)f_{A,b}(3)>f_{A,a+b}(3)
    \end{equation*}
    holds for all pairs $(a,b)\in\mathbb{N}^2$ except for $(1,1)$. In that case, we have that
    \begin{equation*}
        f_{A,1}(3)f_{A,1}(3)\geq f_{A,2}(3)
    \end{equation*}
    with
equality whenever $2\in A$.
\end{thm}
\begin{proof}
    At first, let us notice that for each $n\geq0$, we have $f_{A,n}(3)f_{A,0}(3)=f_{A,n}(3)$. We prove the statement by induction on $a+b$. Let us assume that $a+b=2$. Since
    \begin{align*}
        f_{A,2}(3)=\begin{cases}
9,  & \text{if } 2\in A,\\
6, & \text{if } 2\not\in A,
\end{cases}
    \end{align*}
    it is straightforward to see that
    \begin{align*}
        f_{A,1}(3)f_{A,1}(3)=9\geq f_{A,2}(3).
    \end{align*}
    On the other hand, if $a+b=3$, then
    \begin{align*}
        f_{A,3}(3)=\begin{cases}
22,  & \text{if } 2,3\in A,\\
19, & \text{if } 2\in A \text{ and } 3\not\in A,\\
13, & \text{if } 2\not\in A \text{ and } 3\in A,\\
10, & \text{if } 2,3\not\in A,
\end{cases}
    \end{align*}
    and the inequality 
    \begin{equation*}
        f_{A,1}(3)f_{A,2}(3)>f_{A,3}(3)
    \end{equation*}
    follows. Next, we check the validity of the statement for all pairs $(a,b)$ such that $a,b\leq14$ and possible choices for the set $A$. The procedure which allows us to deal with this task is specifically described in Lemma \ref{Lemma1}. Now, we suppose that the statement is valid for all pairs $(a,b)\in\mathbb{N}^2$ with $3\leq a+b\leq n$, and check whether it is also true for $a+b=n+1.$ Without loss of generality we may assume that $a\geq b\geq1$ and $a\geq15$. Employing the formula (\ref{1}) to the difference $f_{A,a}(3)f_{A,b}(3)-f_{A,a+b}(3)$, we get
    \begin{align*}
        &\frac{3}{a}\sum_{j=1}^a\sigma_A(j)f_{A,a-j}(3)f_{A,b}(3)-\frac{3}{a+b}\sum_{i=1}^{a+b}\sigma_A(i)f_{A,a+b-i}(3)\\
        >&\frac{3}{a}\sum_{j=1}^a\sigma_A(j)f_{A,a-j}(3)f_{A,b}(3)-\frac{3}{a+b}\sum_{i=1}^{a}\sigma_A(i)f_{A,a-i}(3)f_{A,b}(3)\\
        &\phantom{\frac{3}{a}\sum_{j=1}^a\sigma_A(j)f_{A,a-j}(3)f_{A,b}(3)}-\frac{3}{a+b}\sum_{i=1}^{b}\sigma_A(a+i)f_{A,b-i}(3)\\
        =&\frac{3}{a(a+b)}\left(b\sum_{j=1}^a\sigma_A(j)f_{A,a-j}(3)f_{A,b}(3)-a\sum_{i=1}^b\sigma_A(a+i)f_{A,b-i}(3)\right)\\
    \end{align*}
    where the inequality follows directly from the induction hypothesis
which is applied to $f_{A,a+b-i}(3)$ for each $i$ between $1$ and $a$. Thus, it is enough to prove that the expression in the bracket is non-negative. To make the text more transparent, let us set
    \begin{align*}
        \kappa_a:=\sum_{j=1}^a\sigma_A(j)f_{A,a-j}(3).
    \end{align*}
    We obtain that
\begin{align*}
&b\kappa_af_{A,b}(3)-a\sum_{i=1}^b\sigma_A(a+i)f_{A,b-i}(3)\\
    =&3\kappa_a\sum_{i=1}^b\sigma_A(i)f_{A,b-i}(3)-a\sum_{i=1}^b\sigma_A(a+i)f_{A,b-i}(3)\\
    =&\sum_{i=1}^b\left(3\kappa_a\sigma_A(i)-a\sigma_A(a+i)\right)f_{A,b-i}(3).
\end{align*}
Since $1\leq \sigma_A(n)\leq\sigma(n)\leq n(1+\ln{n})$ for every $n\geq1$ and $a\geq b$, we further have that
\begin{align*}
    3\kappa_a-2a^2(1+\ln{(2a)})&\geq3\sum_{j=0}^{a-1}f_{A,j}(3)-2a^2(1+\ln{(2a)})\\
    &\geq3\binom{a+2}{3}-2a^2(1+\ln{(2a)})\label{inequality}\tag{$\star$},
\end{align*}
where the last inequality is a consequence of Lemma \ref{Lemma0}. Hence, it is enough to show that the function 
    \begin{align*}
        \Psi(x):=(x+1)(x+2)-4x(1+\ln{(2x)})
    \end{align*}
    is non-negative for all $x\geq15$; and it is not a difficult problem to solve. The proof is complete by the law of induction.
\end{proof}

For the sake of clarity, let us complete the proof by dealing with the remaining cases for $1\leq a,b\leq14$. 

\begin{lm}\label{Lemma1}
    For any set $A\subseteq\mathbb{N}$ such that $1\in A$, the inequality
    \begin{equation*}
        f_{A,a}(3)f_{A,b}(3)>f_{A,a+b}(3)
    \end{equation*}
    holds for all such pairs $a,b\in\mathbb{N}$ that $a,b\leq14$ and $a+b>2$.
\end{lm}
\begin{proof}
    In order to prove the lemma, we carry out appropriate computations in Wolfram Mathematica \cite{WM}. At the beginning, we assume that $a,b\leq10$, $a+b\geq3$, and consider all possible subsets $A$ of the set $\{1,2,\ldots,a+b\}$ such that $1\in A$. Observe that, in fact, it is enough to consider the aforementioned subsets. Otherwise, if $j\in A$ and $j>a+b$, then $j$ can not appear as a part of the partitions of $a$, $b$ and $a+b$. Hence, any such 
$j$ does not influence the values of $f_{A,a}(3)$, $f_{A,b}(3)$,
and $f_{A,a+b}(3)$. Considering all possible choices of $A$, and numbers $a$ and $b$ such that $a,b\leq10$, we check the correctness of the statement for every such a triple. Since in all of these cases the set $A$ is finite, it follows from Bell's theorem (see, \cite{Bell, KG2}) that the function $f_{A,n}(3)$ is a quasi-polynomial, namely, the expression of the form 
    $$f_{A,n}(3)=c_{l}(n)n^l+c_{l-1}(n)n^{l-1}+\cdots+c_{0}(n),$$
    where $l=3|A|-1$ and all of these coefficients $c_i(n)$ depend on the residue class of $n\pmod*{\lcm(A)}$. Thus, we may explicitly determine its values and directly verify the claim. 
    
    Afterwards, we slightly modify our reasoning for all pairs $a$ and $b$ such that $a\geq b\geq1$ and $11\leq a\leq14$. Mainly because the number of possible choices of the set $A$ rapidly grows with the value of $a+b$, which, on the other hand, significantly influences
the time of the calculations. Thus, the idea is to assume something more about the set $A$. For instance, let us suppose that $1,2\in A$, then $f_{A,n}(3)\geq f_{\{1,2\},n}(3)$ for each non-negative integer $n$. However, we can directly derive the formula for $f_{\{1,2\},n}(3)$ (due to Bell's theorem \cite{Bell}) which takes the form
    \begin{align*}
        f_{\{1,2\},n}(3)=\begin{cases}
    \frac{n^5}{960}+\frac{3n^4}{128}+\frac{19n^3}{96}+\frac{25 n^2}{32}+\frac{173n}{120}+1, & \text{if }  2\mid n, \\
    \frac{n^5}{960}+\frac{3n^4}{128}+\frac{19n^3}{96}+\frac{49n^2}{64}+\frac{1249n}{960}+\frac{91}{128}, & \text{if } 2\nmid n.
  \end{cases}
    \end{align*}
    Now, taking minima of the corresponding coefficients of the above polynomials, it is straightforward to see that
    \begin{align*}
        f_{\{1,2\},n}(3)\geq\frac{n^5}{960}+\frac{3n^4}{128}+\frac{19n^3}{96}+\frac{49n^2}{64}+\frac{1249n}{960}+\frac{91}{128}
    \end{align*}
    holds for every $n\geq0$. Therefore, we might replace the lower bound in (\ref{inequality}) in the proof of Theorem \ref{thm2.1} by the right hand side of the above inequality, and verify that 
    \begin{align*}
        3\sum_{n=1}^{a-1}\left(\frac{n^5}{960}+\frac{3n^4}{128}+\frac{19n^3}{96}+\frac{49n^2}{64}+\frac{1249n}{960}+\frac{91}{128}\right)\geq2a^2(1+\ln{(2a)})
    \end{align*}
    holds for every $a>4$. Actually, we might present a similar reasoning also for $f_{\{1,r\},n}(3)$, where $r\in\{3,4,5,6,7\}$ and obtain that the required inequality for $f_{\{1,r\},n}(3)$ is satisfied for all $a\geq11$ and $b\geq1$. Hence, in order to deal with the remaining cases that are those corresponding to $11\leq a\leq14$ and $b\leq a$, it is enough to repeat the above method for all possible subsets $A$ of $\{1,8,9,\ldots,a+b\}$.
\end{proof}

Before we pass to the central part of this section, let us point out some consequences of Theorem \ref{thm2.1}.

\begin{cor}\label{cor3.5}
    For any finite (or infinite) multiset $B$ such that $1\in B$ and $\mu(j)=3$ for each number $j\in B$, we have that
    \begin{align*}
        p_B(a)p_B(b)>p_B(a+b)
    \end{align*}
    is true for all positive integers $a$ and $b$ (except possibly $a=b=1$ whenever $2\in B$).
\end{cor}

\begin{re}\label{remark3.6}{\rm
     It should be noted that Corollary \ref{cor3.5} significantly improves Gajdzica's result \cite[Theorem 5.4]{KG2} related to the
Bessenrodt--Ono inequality for all multisets $B$ of the form $B=\{1,1,1,a_2,a_2,a_2,\ldots,a_k,a_k,a_k\}$ with $1<a_2<\cdots<a_k$. More precisely, his theorem asserts that
     \begin{align*}
         p_B(a)p_B(b)>p_B(a+b)
     \end{align*}
     holds for all $a,b\geq\frac{2}{3k}\prod_{i=1}^{3k}(1+3ik\cdot\lcm{(a_2,a_3,\ldots,a_k)})+2$. On the other hand, Corollary \ref{cor3.5} maintains that the above inequality is valid for all positive integers $a$ and $b$ (except possibly $a=b=1$ whenever $2\in B$).}
\end{re}

Now, it is time to present the proof of Theorem \ref{th:set}.
\subsection{Proof of Theorem~\ref{th:set}}
\begin{proof}
    We show the claim by induction on $a+b$. For $a+b=2$, we have that 
    \begin{align*}
        f_{A,1}(x)f_{A,1}(x)-f_{A,2}(x)\geq x^2-x/2(x+3)=x(x-3)/2,
    \end{align*}
which is a consequence of the equality
\begin{align*}
        f_{A,2}(x)=\begin{cases}
    x(x+3)/2, & \text{if }  2\in A, \\
    x(x+1)/2, & \text{if } 2\not\in A.
  \end{cases}
    \end{align*}
    For $a+b=3$, on the other hand, one can determine that
    \begin{align*}
        f_{A,3}(x)=\begin{cases}
    x(x+1)(x+8)/6, & \text{if }  2,3\in A, \\
    x\left(x^2+9 x+2\right)/6, & \text{if }  2\in A \text{ and } 3\not\in A, \\
    x\left(x^2+3 x+8\right)/6, & \text{if }  2\not\in A \text{ and } 3\in A, \\
    x(x+1)(x+2)/6, & \text{if }  2,3\not\in A.
  \end{cases}
  \end{align*}
  Hence, we get that
  \begin{align*}
      f_{A,2}(x)f_{A,1}(x)-f_{A,3}(x)=\begin{cases}
    x(x-2)(x+2)/3, & \text{if }  3\in A, \\
    x(x-1)(x+1)/3, & \text{if } 3\not\in A,
  \end{cases}
  \end{align*}
  and the claim follows for $a+b=3$. Thus, let us assume that the statement is valid for all  $3\leq a+b\leq n$, and examine whether it is also correct for $a+b=n+1$ with $a$ and $b$ such that $1\leq b\leq a$ and $a\geq2$. Theorem \ref{thm2.1} points out that the inequality $f_{A,a}(3)f_{A,b}(3)>f_{A,a+b}(3)$ is true for all positive integers $a$ and $b$ such that $a+b>2$. Now, it suffices to show that the derivative
    \begin{align*}
        \frac{d}{dx}\left(f_{A,a}(x)f_{A,b}(x)-f_{A,a+b}(x)\right)
    \end{align*}
    is positive for all $x\geq3$, because this implies that $f_{A,a}(x)f_{A,b}(x)-f_{A,a+b}(x)>0$. Since $f_{A,n}(x)f_{A,0}(x)=f_{A,n}(x)$, we have
    \begin{align*}
        &f_{A,a}'(x)f_{A,b}(x)+f_{A,a}(x)f_{A,b}'(x)-f_{A,a+b}'(x)\\
        =&\sum_{j=1}^a\frac{\sigma_A(j)}{j}f_{A,a-j}(x)f_{A,b}(x)+\sum_{j=1}^b\frac{\sigma_A(j)}{j}f_{A,a}(x)f_{A,b-j}(x)-\sum_{j=1}^{a+b}\frac{\sigma_A(j)}{j}f_{A,a+b-j}(x)\\
        >&\sum_{j=1}^a\frac{\sigma_A(j)}{j}f_{A,a-j}(x)f_{A,b}(x)+\sum_{j=1}^b\frac{\sigma_A(j)}{j}f_{A,a}(x)f_{A,b-j}(x)\\
        -&\sum_{j=1}^{a}\frac{\sigma_A(j)}{j}f_{A,a-j}(x)f_{A,b}(x)-\sum_{j=1}^b\frac{\sigma_A(a+j)}{a+j}f_{A,b-j}(x)\\
        =&\sum_{j=1}^b\left(\frac{\sigma_A(j)}{j}f_{A,a}(x)-\frac{\sigma_A(a+j)}{a+j}\right)f_{A,b-j}(x).
    \end{align*}
    Therefore, we can just show that
    \begin{align*}
        \frac{\sigma_A(j)}{j}f_{A,a}(x)-\frac{\sigma_A(a+j)}{a+j}\geq0
    \end{align*}
    for each $j$ between $1$ and $b$. Since $a\geq b$ and $1\leq\sigma_A(n)\leq\sigma(n)\leq n(1+\ln{n})$ for each $n\geq1$, it is easy to see that the inequality
    \begin{align*}
        \frac{\sigma_A(j)}{j}f_{A,a}(x)-\frac{\sigma_A(a+j)}{a+j}\geq\frac{1}{a}\binom{x+a-1}{a}-(1+\ln{(2a)})
    \end{align*}
    is valid for every such a $j$. However, the assumption that $x\geq3$ asserts that
    \begin{align*}
        \frac{1}{a}\binom{x+a-1}{a}-(1+\ln{(2a)})\geq (a+1)(a+2)/(2a)-1-\ln{(2a)}.
    \end{align*}
    Thus, it suffices to prove that the function 
    \begin{align*}
        \psi(x):=(x+1)(x+2)-2x(1+\ln{(2x)})
    \end{align*}
    is positive for all real numbers $x\geq3$, and
this is not difficult to verify.
\end{proof}

Since Theorem \ref{th:set} is a very general result, it delivers us a lot of meaningful consequences.

\begin{cor}
For the set $A_{III}=\mathbb{N}$, the inequality
    \begin{equation*}
        f_{A_{III},a}(x)f_{A_{III},b}(x)>f_{A_{III},a+b}(x)
    \end{equation*}
    is valid for all $x\geq3$ and $a,b\in\mathbb{N}$ except $a=b=1$. For $a=b=1$, we have that 
    $$f_{A_{III},1}(3)f_{A_{III},1}(3)=f_{A_{III},2}(3) \hspace{0.2cm} \text{and}\hspace{0.2cm} f_{A_{III},1}(x)f_{A_{III},1}(x)>f_{A_{III},2}(x)$$
    for all $x>3$.
\end{cor}
One can compare the obtained result with the theorem due to Heim, Neuhauser,
and Tr\"{o}ger \cite[Theorem 1.4]{HN6}. More specifically, they proved that the inequality
    \begin{equation*}
        f_{A_{III},a}(x)f_{A_{III},b}(x)>f_{A_{III},a+b}(x) 
    \end{equation*}
    is satisfied for all $x>2$, $a,b\in\mathbb{N}$ and $a+b>2$. For $x=2$, it is true whenever $a+b>4$. Clearly, Theorem \ref{th:set} has some additional requirements on the values of $a,b$ and $x$, which is a consequence of its general character.

\begin{re}
    Let $k\geq2$ be a natural number and let $A=\mathbb{N}\setminus \{jk:j\in\mathbb{N}\}$. Theorem \ref{th:set} asserts the polynomization of the
Bessenrodt--Ono inequality for the $k$-regular partition function $p_k(n)$ for any such a $k$. Recall that Beckwith and Bessenrodt \cite[Theorem 2.1]{BB} proved the
Bessenrodt--Ono inequality for $p_k(n)$ for $2\leq k \leq6$. Hence, to resolve the issue completely, it suffices to find the analogue of their result for any $k\geq 7$ and investigate the $2$-colored $k$-regular partition function in that regard.
\end{re}

At the end of this section, let us show the general version of Corollary \ref{cor3.5}.

\begin{cor}
    Let $m>3$ be a fixed natural number, and let $B$ be an arbitrary multiset of positive integers such that $1\in B$ and $\mu(j)=m$ for every $j\in B$. Then the inequality
    \begin{equation*}
        p_B(a)p_B(b)>p_B(a+b)
    \end{equation*}
    is true for all positive integers $a$ and $b$.
\end{cor}

\section{The Bessenrodt--Ono type inequality for the A-plane partitions}

In this section we have two general assumptions on the multiset $A$. The first
one is that $1$ occurs exactly once
in $A$. The second one, on the other hand, states that each number $j$ might appear at most $j$ times in $A$---in
other words $\mu(j)\leq j$. In particular, observe that if every number $j$ occurs exactly $j$ times in $A$, then the corresponding $A$-partition function is just the plane partition function $p_A(n)=\op{pp}\left( n\right) $,
and that is the reason for the title of this section. Nevertheless, there is a myriad possibilities to choose the multiset $A$. We can, for instance, consider any set $A$ which satisfies the general assumptions from Section 3. However, we do not need to restrict ourselves only to sets.

Let us assume that $A=\{1,2,2,3,5,5,5\}$. The corresponding polynomials $f_{A,n}(x)$ might be easily calculated by employing the formula $(\ref{1})$. Table 2 presents these polynomials for $1\leq n\leq7$.
    \begin{center}
    \begin{table}
\begin{tabular}{ |c|l| } 
\hline
$n$ & $f_{A,n}(x)$\\
 \hline\hline 
 $1$ & $x$ \\ 
 $2$ & $\frac{1}{2} x (x+5)$ \\ 
 $3$ & $\frac{1}{6} x \left(x^2+15 x+8\right)$  \\ 
 $4$ & $\frac{1}{24} x \left(x^3+30 x^2+107 x+30\right)$  \\ 
 $5$ & $\frac{1}{120} x \left(x^4+50 x^3+455 x^2+550 x+384\right)$ \\ 
 $6$ & $\frac{1}{720} x \left(x^5+75 x^4+1285 x^3+4725 x^2+5194 x+960\right)$  \\ 
 $7$ & $\frac{1}{5040}x \left(x^6+105 x^5+2905 x^4+22575 x^3+49294 x^2+55440 x+720\right)$ \\
 \hline
\end{tabular}
\caption{The polynomials $f_{A,n}(x)$ for $A=\{1,2,2,3,5,5,5\}$ and $1\leq n \leq7$.}
\end{table}
\end{center}

Once again, one can ask about the localization of the complex roots of the polynomials $f_{A,a}(x)f_{A,b}(x)-f_{A,a+b}(x)$. Figure 2 shows all their complex roots for $1\leq a,b \leq10$.

\begin{center}
\begin{figure}[!htb]
   \begin{minipage}{1\textwidth}
     \centering
     \includegraphics[width=1\linewidth]{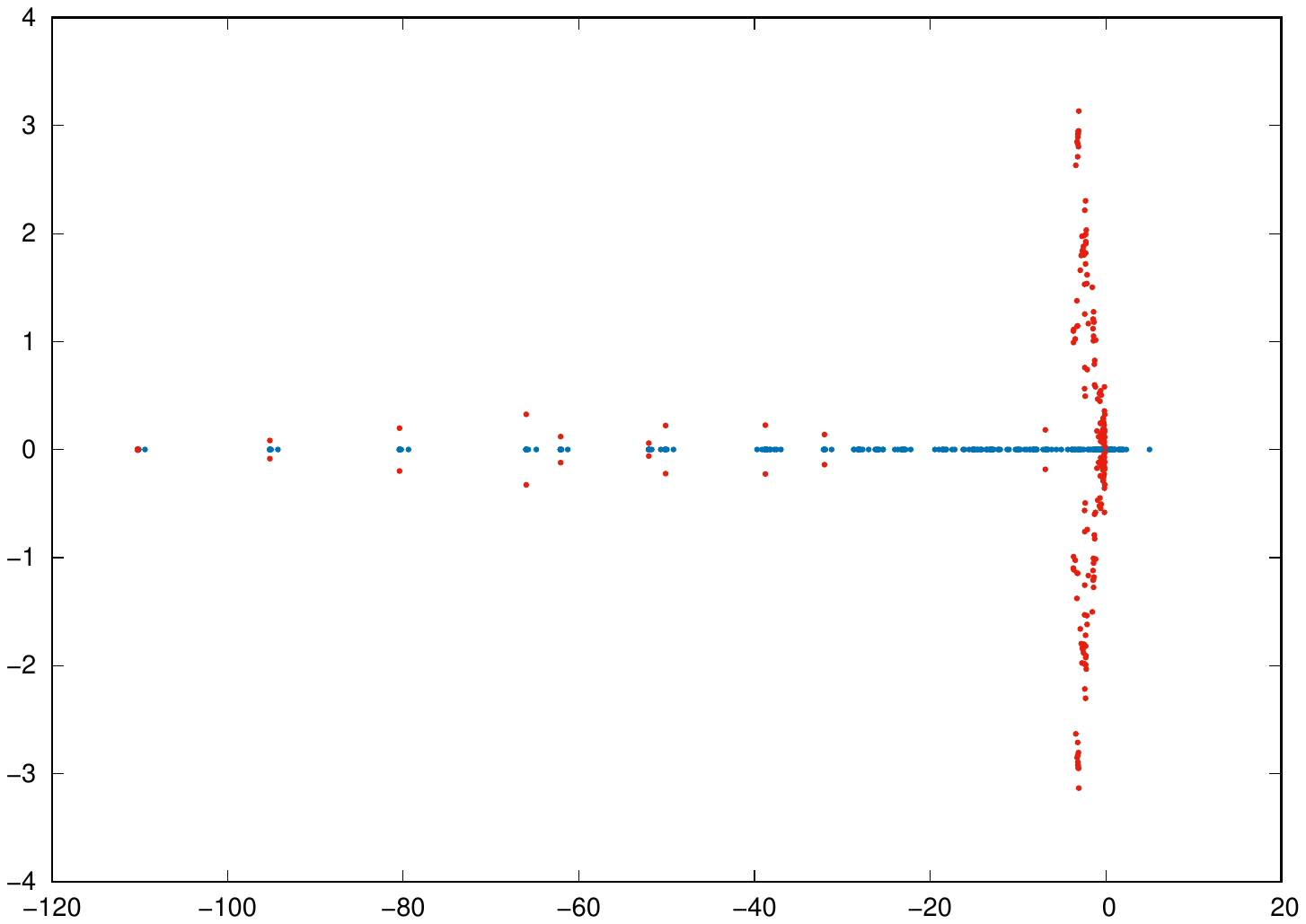}
     \caption{Complex roots of the polynomials $f_{A,a}(x)f_{A,b}(x)-f_{A,a+b}(x)$ for $A=\{1,2,2,3,5,5,5\}$ and $1\leq a,b \leq10$.}
   \end{minipage}\hfill
\end{figure}
\end{center}

Now, we wish to go to the main part of the section and find out the analogue of Theorem \ref{thm2.1}. However, in that case we consider the $5$-colored $A$-partition function which agrees with the result due to Heim, Neuhauser,
and Tr\"{o}ger \cite{HNT2} for the classical plane partition function.

\begin{thm}\label{thm4.1}
    If $A$ is such a multiset of natural numbers that $1$ occurs exactly once
in $A$, and each number $j$ appears at most $j$ times in $A$. Then, the inequality
    \begin{equation*}
        f_{A,a}(5)f_{A,b}(5)>f_{A,a+b}(5)
    \end{equation*}
    holds for all $a,b\in\mathbb{N}$ with $a+b>1$. For $a+b=1$, we have that
    \begin{equation*}
        f_{A,1}(5)f_{A,1}(5)\geq f_{A,2}(5)
    \end{equation*}
    with the equality whenever $2$ occurs exactly two times in $A$.
\end{thm}
\begin{proof}
    As before, the proof is by the induction on $a+b$. If $a+b=2$, then it is not difficult to check that
    \begin{align*}
        f_{A,2}(5)=\begin{cases}
        25, & \text{if }  2 \text{ appears exactly two times in }A,\\
        20, & \text{if }  2 \text{ appears exactly one time in }A, \\
    15, & \text{if }  2\not\in A,
  \end{cases}
    \end{align*}
    and, thus,
    \begin{align*}
        f_{A,1}(5)f_{A,1}(5)=25\geq f_{A,2}(5),
    \end{align*}
    where the equality occurs if and only if $2$ appears exactly two times in $A$. Let us also verify the statement for $a+b=3$.
    In that case, it is not difficult to check that
    \begin{align*}
        f_{A,3}(5)\leq\begin{cases}
        100, & \text{if }  2 \text{ appears exactly two times in }A,\\
        75, & \text{if }  2 \text{ appears exactly one time in }A, \\
    50, & \text{if }  2\not\in A.
  \end{cases}
    \end{align*}
    Therefore, we immediately obtain that 
    \begin{equation*}
        f_{A,2}(5)f_{A,1}(5)>f_{A,3}(5),
    \end{equation*}
    as required. Now, we use the same methods as in Lemma \ref{Lemma1} to verify the claim for all $a,b\leq8$ and possible choices for $A$. In fact, it turns out that the statement is valid for all the cases. Hence, let us assume that the claim is true for all pairs $a$ and $b$ such that $3\leq a+b\leq n$, and check whether it is also true for $a+b=n+1.$ Without loss of generality we may additionally require that $a\geq b\geq1$ and $a\geq9$.  Once again, applying the formula (\ref{1}) to the difference $f_{A,a}(5)f_{A,b}(5)-f_{A,a+b}(5)$, we get
    \begin{align*}
        &\frac{5}{a}\sum_{j=1}^a\sigma_A(j)f_{A,a-j}(5)f_{A,b}(5)-\frac{5}{a+b}\sum_{i=1}^{a+b}\sigma_A(i)f_{A,a+b-i}(5)\\
        >&\frac{5}{a}\sum_{j=1}^a\sigma_A(j)f_{A,a-j}(5)f_{A,b}(5)-\frac{5}{a+b}\sum_{i=1}^{a}\sigma_A(i)f_{A,a-i}(5)f_{A,b}(5)\\
        &\phantom{\frac{5}{a}\sum_{j=1}^a\sigma_A(j)f_{A,a-j}(5)f_{A,b}(5)}-\frac{5}{a+b}\sum_{i=1}^{b}\sigma_A(a+i)f_{A,b-i}(5)\\
        =&\frac{5}{a(a+b)}\left(b\sum_{j=1}^a\sigma_A(j)f_{A,a-j}(5)f_{A,b}(5)-a\sum_{i=1}^b\sigma_A(a+i)f_{A,b-i}(5)\right),\\
    \end{align*}
    where the second line follows from the induction hypothesis
applied to $f_{A,a+b-i}(5)$ for each $i$ between $1$ and $a$. Thus, it suffices to show that the expression in the bracket is non-negative. Let us put
    \begin{align*}
        \kappa_a:=\sum_{j=1}^a\sigma_A(j)f_{A,a-j}(5).
    \end{align*}
    We see that
\begin{align*}
&b\kappa_af_{A,b}(5)-a\sum_{i=1}^b\sigma_A(a+i)f_{A,b-i}(5)\\
    =&5\kappa_a\sum_{i=1}^b\sigma_A(i)f_{A,b-i}(5)-a\sum_{i=1}^b\sigma_A(a+i)f_{A,b-i}(5)\\
    =&\sum_{i=1}^b\left(5\kappa_a\sigma_A(i)-a\sigma_A(a+i)\right)f_{A,b-i}(5).
\end{align*}
Since $1\leq\sigma_A(n)\leq\sigma_2(n)\leq n(2n-1)$ for each $n\geq1$ and $a\geq b$, we further simplify the above to
\begin{align*}
    5\kappa_a-2a^2(4a-1)\geq5\sum_{j=0}^{a-1}f_{A,j}(5)-2a^2(4a-1)\geq5\binom{a+4}{5}-2a^2(4a-1),
\end{align*}
where the last inequality is the consequence of Lemma \ref{Lemma0}. Therefore, it is enough to prove that the function 
    \begin{align*}
        \Psi(x):=(x+1)(x+2)(x+3)(x+4)-48x(4x-1)
    \end{align*}
    is non-negative for all $x\geq9$ which is not a difficult task. This ends the proof by the law of induction.
\end{proof}

There are a few direct consequences of Theorem \ref{thm4.1}. Since they are very similar to Corollary \ref{cor3.5} and Remark \ref{remark3.6}, we omit them and proceed to the proof of Theorem \ref{th:multiset}.

\subsection{Proof of Theorem~\ref{th:multiset}}
\begin{proof}
    We show the claim by induction on $a+b$. For $a+b=2$, we have that 
    \begin{align*}
        f_{A,1}(x)f_{A,1}(x)-f_{A,2}(x)\geq x^2-x/2(x+5)=x(x-5)/2,
    \end{align*}
which is a consequence of the equality
\begin{align*}
        f_{A,2}(x)=\begin{cases}
        x(x+5)/2, & \text{if }  2\text{ appears exactly two times in }A, \\
    x(x+3)/2, & \text{if }  2\text{ appears exactly one time in }A, \\
    x(x+1)/2, & \text{if } 2\not\in A.
  \end{cases}
    \end{align*}
    For $a+b=3$, on the other hand, one can verify that
  \begin{align*}
      f_{A,2}(x)f_{A,1}(x)-f_{A,3}(x)=\begin{cases}
      x\left(x^2-10\right)/3, & \text{if }  3\text{ occurs three times in }A, \\
      x\left(x^2-7\right)/3, & \text{if }  3\text{ occurs two times in }A, \\
    x(x-2)(x+2)/3, & \text{if }  3\text{ occurs one time in }A, \\
    x(x-1)(x+1)/3, & \text{if } 3\not\in A,
  \end{cases}
  \end{align*}
  and we see that the claim is true for $a+b=3$. Thus, let us suppose that the statement is valid for all  $3\leq a+b\leq n$, and investigate whether it is also correct for $a+b=n+1$ with $a$ and $b$ such that $1\leq b\leq a$ and $a\geq2$. Theorem \ref{thm4.1} asserts that the inequality $f_{A,a}(5)f_{A,b}(5)>f_{A,a+b}(5)$ is true for all positive integers $a$ and $b$ such that $a+b>2$. Now, it suffices to show that the derivative
    \begin{align*}
        \frac{d}{dx}\left(f_{A,a}(x)f_{A,b}(x)-f_{A,a+b}(x)\right)
    \end{align*}
    is positive for all $x\geq5$, because it maintains that $f_{A,a}(x)f_{A,b}(x)-f_{A,a+b}(x)>0$. Since $f_{A,n}(x)f_{A,0}(x)=f_{A,n}(x)$, we get
    \begin{align*}
        &f_{A,a}'(x)f_{A,b}(x)+f_{A,a}(x)f_{A,b}'(x)-f_{A,a+b}'(x)\\
        =&\sum_{j=1}^a\frac{\sigma_A(j)}{j}f_{A,a-j}(x)f_{A,b}(x)+\sum_{j=1}^b\frac{\sigma_A(j)}{j}f_{A,a}(x)f_{A,b-j}(x)-\sum_{j=1}^{a+b}\frac{\sigma_A(j)}{j}f_{A,a+b-j}(x)\\
        >&\sum_{j=1}^a\frac{\sigma_A(j)}{j}f_{A,a-j}(x)f_{A,b}(x)+\sum_{j=1}^b\frac{\sigma_A(j)}{j}f_{A,a}(x)f_{A,b-j}(x)\\
        -&\sum_{j=1}^{a}\frac{\sigma_A(j)}{j}f_{A,a-j}(x)f_{A,b}(x)-\sum_{j=1}^b\frac{\sigma_A(a+j)}{a+j}f_{A,b-j}(x)\\
        =&\sum_{j=1}^b\left(\frac{\sigma_A(j)}{j}f_{A,a}(x)-\frac{\sigma_A(a+j)}{a+j}\right)f_{A,b-j}(x).
    \end{align*}
    Therefore, it suffices to prove that
    \begin{align*}
        \frac{\sigma_A(j)}{j}f_{A,a}(x)-\frac{\sigma_A(a+j)}{a+j}\geq0
    \end{align*}
    for each $j$ between $1$ and $b$. Since $a\geq b$ and $1\leq\sigma_A(n)\leq\sigma_2(n)\leq n(2n-1)$ for every $n\geq1$, it is clear that
    \begin{align*}
        \frac{\sigma_A(j)}{j}f_{A,a}(x)-\frac{\sigma_A(a+j)}{a+j}\geq\frac{1}{a}\binom{x+a-1}{a}-4a+1
    \end{align*}
    is valid for each such a $j$. However, the assumption that $x\geq5$ points out that
    \begin{align*}
        \frac{1}{a}\binom{x+a-1}{a}-4a+1\geq (a+1)(a+2)(a+3)(a+4)/(24a)-4a+1.
    \end{align*}
    Now, we just examine when the function 
    \begin{align*}
        \psi(x):=(x+1)(x+2)(x+3)(x+4)-96x^2+24x
    \end{align*}
    is positive. One can easily check that it is positive for every positive real number $x$. This completes the proof by the law of induction. 
\end{proof}

It is worth pointing out that Theorem \ref{th:multiset} agrees with the result of Heim, Neuhauser,
and Tr\"{o}ger \cite[Theorem 1.3]{HNT2}, for the polynomization of the plane partition function $\op{pp}\left( n\right) $.
Actually, Theorem \ref{th:multiset} is a general property which might be successfully applied for various multisets $A$. In particular, it extends Theorem \ref{th:set} for all $x\geq5$.

Nevertheless, there is still a wide class of multisubsets for which we can utilize neither Theorem \ref{th:set} nor Theorem \ref{th:multiset}, and that
gap to our current knowledge will be investigated in the future.

\section*{Acknowledgements}
The first author would like to thank Piotr Miska, Maciej Ulas,
and Błażej Żmija for their valuable comments and helpful suggestions. His research was funded by both a grant of the National Science Centre (NCN), Poland, no. UMO-2019/34/E/ST1/00094 and a grant from the Faculty of Mathematics and Computer Science under the Strategic Program Excellence Initiative at the Jagiellonian University in Kraków.

\end{document}